\documentclass[final,leqno]{siamltex}
\usepackage{amsmath}
\usepackage{amssymb}
\usepackage{euscript}
\usepackage{mathrsfs} 
\renewcommand{\leq}{\ensuremath{\leqslant}}
\renewcommand{\geq}{\ensuremath{\geqslant}}

\newcommand{\minimize}[2]{\ensuremath{\underset{\substack{{#1}}}%
{\text{minimize}}\;\;#2 }}

\newcommand{\Frac}[2]{\displaystyle{\frac{#1}{#2}}} 
 
\newcommand{\scal}[2]{{\langle{{#1}\mid{#2}}\rangle}}
\newcommand{\sscal}[2]{{\big\langle{{#1}\mid{#2}}\big\rangle}}

\newcommand{\menge}[2]{\big\{{#1}~\big |~{#2}\big\}} 
\newcommand{\Menge}[2]{\left\{{#1}~\Big|~{#2}\right\}} 
\newcommand{\KKK}{\ensuremath{\boldsymbol{\mathcal K}}}
\newcommand{\HHH}{\ensuremath{\boldsymbol{\mathcal H}}}

\newcommand{\HH}{\ensuremath{{\mathcal H}}}
\newcommand{\GG}{\ensuremath{{\mathcal G}}}

\newcommand{\Sum}{\ensuremath{\displaystyle\sum}}
\newcommand{\emp}{\ensuremath{{\varnothing}}}

\newcommand{\Id}{\ensuremath{\operatorname{Id}}\,}
\newcommand{\cart}{\ensuremath{\raisebox{-0.5mm}{\mbox{\LARGE{$\times$}}}}}

\newcommand{\RR}{\ensuremath{\mathbb{R}}}

\newcommand{\RPP}{\ensuremath{\left]0,+\infty\right[}}

\newcommand{\RX}{\ensuremath{\left]-\infty,+\infty\right]}}

\newcommand{\NN}{\ensuremath{\mathbb N}}

\newcommand{\weakly}{\ensuremath{\:\rightharpoonup\:}}
\newcommand{\exi}{\ensuremath{\exists\,}}
\newcommand{\ran}{\ensuremath{\text{\rm ran}\,}}

\newcommand{\pinf}{\ensuremath{{+\infty}}}

\newcommand{\dom}{\ensuremath{\text{\rm dom}\,}}
\newcommand{\prox}{\ensuremath{\text{\rm prox}}}

\newcommand{\gra}{\ensuremath{\text{\rm gra}\,}}

\newcommand{\reli}{\ensuremath{\text{\rm ri}\,}}
\newcommand{\zeroun}{\ensuremath{\left]0,1\right[}}   
   
\newtheorem{example}[theorem]{Example}
\newtheorem{problem}[theorem]{Problem}
\newtheorem{remark}[theorem]{Remark}
\renewcommand\theenumi{(\roman{enumi})}
\renewcommand\theenumii{(\alph{enumii})}

\title{Solving Coupled Composite Monotone Inclusions by Successive 
Fej\'er Approximations of Their Kuhn-Tucker Set\thanks{Received by 
the editors XXX XX, 2013; accepted for publication XXX XX, 2014;
published electronically DATE. \URL mms/x-x/XXXX.html.}}

\author{Abdullah Alotaibi\thanks{King Abdulaziz University,
Department of Mathematics, Jeddah 21859, Saudi Arabia
(aalotaibi@kau.edu.sa)}
\and
Patrick L. Combettes,\thanks{Sorbonne Universit\'es -- 
UPMC Univ.\ Paris 06, 
UMR 7598, Laboratoire Jacques-Louis Lions,
F-75005, Paris, France (plc@ljll.math.upmc.fr)} 
\and 
Naseer Shahzad\thanks{King Abdulaziz University,
Department of Mathematics, Jeddah 21859, Saudi Arabia
(nshahzad@kau.edu.sa)}
}

\begin{document}

\maketitle
\newcommand{\slugmaster}{\slugger{siopt}{xxxx}{xx}{x}{x--x}}%

\begin{abstract} 
We propose a new class of primal-dual Fej\'er monotone algorithms for
solving systems of composite monotone inclusions. Our construction 
is inspired by a framework used by Eckstein and Svaiter for the basic
problem of finding a zero of the sum of two monotone operators. At
each iteration, points in the graph of the monotone operators
present in the model are used to construct a half-space containing
the Kuhn-Tucker set associated with the system. The primal-dual
update is then obtained via a relaxed projection of the current
iterate onto this half-space. An important feature that
distinguishes the resulting splitting algorithms from existing ones
is that they do not require prior knowledge of bounds on the 
linear operators involved or the inversion of linear operators.
\end{abstract} 



\maketitle

\section{Introduction}

The first monotone operator splitting methods arose in the late
1970s and were motivated by applications in mechanics and partial
differential equations \cite{Glow83,Glow89,Merc80}.
In recent years, the field of monotone operator splitting 
algorithms has benefited from a new impetus, fueled by emerging
application areas such as signal and image processing, statistics,
optimal transport, machine learning, and domain decomposition 
methods \cite{Atto11,Bach12,Banf11,Gold12,Papa13,Ragu13,Wrig12}. 
Three main algorithms dominate the field explicitly or 
implicitly: the forward-backward method \cite{Merc79},
the Douglas-Rachford method \cite{Lion79}, and the
forward-backward-forward method \cite{Tsen00}. These methods were
originally designed to solve inclusions of the type $0\in Ax+Bx$, 
where $A$ and $B$ are maximally monotone operators acting on a 
Hilbert space (via product space reformulations, they can also
be extended to problems involving sums of more than 2 
operators \cite{Livre1,Spin83}). Until recently, a significant
challenge in the field was to design splitting techniques for 
inclusions involving linearly composed operators, say 
\begin{equation}
\label{eXc5t-iU8-20p}
0\in Ax+L^*BLx, 
\end{equation}
where $A$ and $B$ are maximally monotone 
operators acting on Hilbert spaces $\HH$ and $\GG$, respectively, 
and $L$ is a bounded linear operator from $\HH$ to $\GG$. In the
case when $A$ and $B$ are subdifferentials, say $A=\partial f$ 
and $B=\partial g$, where $f\colon\HH\to\RX$ and 
$g\colon\GG\to\RX$ are lower semicontinuous 
convex functions satisfying a suitable constraint qualification,
\eqref{eXc5t-iU8-20p} corresponds to the 
minimization problem 
\begin{equation}
\label{eXc5t-iU8-20a}
\minimize{x\in\HH}{f(x)+g(Lx)}.
\end{equation}
The Fenchel-Rockafellar dual of this problem is
\begin{equation}
\label{eXc5t-iU8-20b}
\minimize{v^*\in\GG}{f^*(-L^*v^*)+g^*(v^*)}
\end{equation}
and the associated Kuhn-Tucker set is 
\begin{equation}
\label{eXc5t-iU8-20c}
\boldsymbol{Z}=\menge{(x,v^*)\in\HH\oplus\GG}
{-L^*v^*\in\partial f(x)\:\;\text{and}\;Lx\in \partial g^*(v^*)}.
\end{equation}
The importance of this set is discussed extensively in 
\cite{Rock74}, notably in connection with the fact that 
Kuhn-Tucker points provide solutions to \eqref{eXc5t-iU8-20a} 
and \eqref{eXc5t-iU8-20b}.
To the best of our knowledge, the first splitting method for
composite problems of the form \eqref{eXc5t-iU8-20p} is that 
proposed in \cite{Siop11}, which was developed around the following 
formulation.

\begin{problem}
\label{prob:1}
\rm
Let $\HH$ and $\GG$ be real Hilbert spaces, and set 
$\KKK=\HH\oplus\GG$. Let 
$A\colon\HH\to 2^{\HH}$ and $B\colon\GG\to 2^{\GG}$ be maximally 
monotone operators, and let $L\colon\HH\to\GG$ be a bounded linear 
operator. Consider the inclusion problem
\begin{equation}
\label{e:primal}
\text{find}\;\;\overline{x}\in\HH\;\;\text{such that}\;\; 
0\in A\overline{x}+L^*BL\overline{x},
\end{equation}
the dual problem
\begin{equation}
\label{e:dual}
\text{find}\;\;\overline{v}^*\in\GG\;\;\text{such that}\;\; 
0\in -LA^{-1}(-L^*\overline{v}^*)+B^{-1}\overline{v}^*,
\end{equation}
and the associated Kuhn-Tucker set
\begin{equation}
\label{aqYg6521a}
\boldsymbol{Z}=\menge{(x,v^*)\in\KKK}
{-L^*v^*\in Ax\:\;\text{and}\;Lx\in B^{-1}v^*}.
\end{equation}
The problem is to find a point in $\boldsymbol{Z}$.
The sets of solutions to \eqref{e:primal} and \eqref{e:dual} are 
denoted by $\mathscr{P}$ and $\mathscr{D}$, respectively.
\end{problem}

The Kuhn-Tucker set \eqref{aqYg6521a} is a natural extension of 
\eqref{eXc5t-iU8-20c} to general monotone operators.
In \cite{Siop11}, a point in $\boldsymbol{Z}$ was obtained 
by applying the forward-backward-forward method to a suitably 
decomposed inclusion in $\HH\oplus\GG$ (the use of 
Douglas-Rachford splitting was also discussed there). Subsequently, 
the idea of using traditional splitting techniques to find 
Kuhn-Tucker points was further exploited in a variety of settings, 
e.g., \cite{Optl14,Bot13a,Bot13c,Siop13,Svva12,Opti14,Bang13}. 
Despite their broad range of applicability, existing splitting 
methods suffer from two shortcomings that precludes their use in
certain settings. Thus, a shortcoming of splitting methods 
based on the forward-backward-forward \cite{Siop11,Svva12} or the 
forward-backward algorithms \cite{Sico10,Opti14,Bang13}
is that they require knowledge of $\|L\|$; this is also true for
the Douglas-Rachford-based method of \cite{Bot13c}.
On the other hand, a shortcoming of splitting methods based on 
the Douglas-Rachford \cite[Remark~2.9]{Siop11} or Spingarn 
\cite{Optl14} algorithms is that they require the inversion of 
linear operators, as does \cite[Algorithm~3]{Bot13a}. In some
applications, however, $\|L\|$ cannot be evaluated reliably and the
inversion of linear operators is not numerically feasible.
As will be seen in Section~\ref{sec:4}, this issue becomes 
particularly acute when dealing with systems of coupled monotone 
inclusions, which constitute the main motivation for
our investigation.

Our objective is to devise a new class of algorithms for solving 
Problem~\ref{prob:1} that alleviate the above-mentioned 
shortcomings of existing methods. Our approach is inspired by an
original splitting framework proposed in \cite{Svai08} for solving
the basic inclusion (see also \cite{Svai09} for 
the extension to the sum of several operators)
\begin{equation}
\label{HJG7t6zzv02-20s}
0\in Ax+Bx. 
\end{equation}
The main idea of \cite{Svai08} is to use points in the graphs of
$A$ and $B$ to construct a sequence of Fej\'er approximations to
the so-called extended solution set 
\begin{equation}
\menge{(x,v^*)\in\HH\oplus\HH}
{-v^*\in Ax\:\;\text{and}\;v^*\in Bx}
\end{equation}
and to iterate by
projection onto these successive approximations. This extended 
solution set is actually nothing but the specialization of the
Kuhn-Tucker set 
\eqref{aqYg6521a} to the case when $\GG=\HH$ and $L=\Id$. 
This construction led to novel splitting methods for solving
\eqref{HJG7t6zzv02-20s} that do not seem to 
derive from the traditional methods mentioned above. In the present
paper, we extend it significantly beyond 
\eqref{HJG7t6zzv02-20s} in order to design new primal-dual 
splitting algorithms for Problem~\ref{prob:1}. 

The paper is organized as follows. Preliminary results are 
established in Section~\ref{sec:2} and algorithms for solving 
Problem~\ref{prob:1} are developed in Section~\ref{sec:3}. 
These results are then used in Section~\ref{sec:4} to solve 
systems of composite monotone inclusions in duality. 

\noindent
{\bfseries Notation.}
The scalar product of a Hilbert space is denoted by 
$\scal{\cdot}{\cdot}$ and the associated norm by $\|\cdot\|$.
The symbols $\weakly$ and $\to$ denote, respectively, weak and 
strong convergence, and $\Id$ denotes the identity operator. 
Let $\HH$ and $\GG$ be real Hilbert spaces, let $2^{\HH}$ be the 
power set of $\HH$, and let $A\colon\HH\to 2^{\HH}$. We denote by 
$\ran A=\menge{u\in\HH}{(\exi x\in\HH)\;u\in Ax}$ the range of $A$, 
by $\gra A=\menge{(x,u)\in\HH\times\HH}{u\in Ax}$ the graph of 
$A$, and by $A^{-1}$ the inverse of $A$, which is defined through
its graph $\menge{(u,x)\in\HH\times\HH}{(x,u)\in\gra A}$. 
The resolvent of $A$ is $J_A=(\Id+A)^{-1}$. We say that $A$ is 
monotone if
\begin{equation}
(\forall (x,u)\in\gra A)(\forall (y,v)\in\gra A)\quad
\scal{x-y}{u-v}\geq 0,
\end{equation}
and maximally monotone if there does not exist any monotone operator 
$B\colon\HH\to 2^{\HH}$ such that $\gra A\subset\gra B\neq\gra A$.
In this case, $J_A$ is firmly nonexpansive and defined everywhere
on $\HH$. The Hilbert direct sum of $\HH$ and $\GG$ is denoted by
$\HH\oplus\GG$. The projection operator onto a nonempty closed 
convex subset $C$ of $\HH$ is denoted by $P_C$. The necessary 
background on convex analysis and monotone operators will be found 
in \cite{Livre1}.

\section{Preliminary results}
\label{sec:2}

We first investigate some basic properties of Problem~\ref{prob:1}, 
starting with the fact that Kuhn-Tucker points automatically 
provide primal and dual solutions.

\begin{proposition}
\label{p:uuYt6y31z}
In the setting of Problem~\ref{prob:1}, the following hold:
\begin{enumerate}
\item
\label{p:uuYt6y31zi}
$\boldsymbol{Z}$ is a closed convex subset of 
$\mathscr{P}\times\mathscr{D}$. 
\item
\label{p:uuYt6y31zii}
$\mathscr{P}\neq\emp$ $\Leftrightarrow$
$\boldsymbol{Z}\neq\emp$ $\Leftrightarrow$
$\mathscr{D}\neq\emp$.
\end{enumerate}
\end{proposition}
\begin{proof}
This is {\rm\cite[Proposition~2.8]{Siop11}}
(see also \cite{Penn00} for \ref{p:uuYt6y31zii}).
\end{proof}

A fundamental concept in algorithmic nonlinear analysis is that
of Fej\'er monotonicity: a sequence $(\boldsymbol{x}_n)_{n\in\NN}$ 
in a Hilbert space $\HHH$ is said to be Fej\'er monotone with 
respect to a set $\boldsymbol{C}\subset\HHH$ if 
\begin{equation}
\label{ecnw7Gjh1120f}
(\forall{\boldsymbol{z}}\in{\boldsymbol{C}})(\forall n\in\NN)\quad
\|{\boldsymbol{x}}_{n+1}-{\boldsymbol{z}}\|\leq
\|{\boldsymbol{x}}_n-{\boldsymbol{z}}\|.
\end{equation}
Alternatively (see \cite[Section~2]{Moor01}), 
$(\boldsymbol{x}_n)_{n\in\NN}$ is Fej\'er monotone 
with respect to $\boldsymbol{C}$ if, for every $n\in\NN$, 
$\boldsymbol{x}_{n+1}$ is a relaxed projection of $\boldsymbol{x}_n$
onto a closed affine half-space $\boldsymbol{H}_n$ containing
$\boldsymbol{C}$, i.e.,
\begin{equation}
\label{ecnw7Gjh1120g}
(\forall n\in\NN)\quad \boldsymbol{x}_{n+1}=\boldsymbol{x}_n+
\lambda_n(P_{\boldsymbol{H}_n}\boldsymbol{x}_n-\boldsymbol{x}_n),
\quad\text{where}\quad 0\leq\lambda_n\leq2\quad\text{and}\quad
\boldsymbol{C}\subset\boldsymbol{H}_n.
\end{equation}
The half-spaces $(\boldsymbol{H}_n)_{n\in\NN}$ in 
\eqref{ecnw7Gjh1120g} are called Fej\'er approximations to 
$\boldsymbol{C}$. The Fej\'er monotonicity 
property \eqref{ecnw7Gjh1120f} makes it possible to greatly 
simplify the analysis of the asymptotic behavior of a broad class 
of algorithms; see \cite{Bau96a,Livre1,Aiep96,Eoop01,Ere68a,Erem09}
for background, examples, and historical notes.  

In the following proposition, we consider the problem of
constructing a Fej\'er approximation to the 
Kuhn-Tucker set \eqref{aqYg6521a}.

\begin{proposition}
\label{77htFq10}
In the setting of Problem~\ref{prob:1},
for every $\mathsf{a}=(a,a^*)\in\gra A$ and
$\mathsf{b}=(b,b^*)\in\gra B$, set
\begin{equation}
\label{kL8u8625}
\boldsymbol{H}_{\mathsf{a},\mathsf{b}}
=\Menge{\boldsymbol{x}\in\KKK}
{\scal{\boldsymbol{x}}{\boldsymbol{s}^*_{\mathsf{a},
\mathsf{b}}}\leq\eta_{\mathsf{a},\mathsf{b}}},
\quad\text{where}\quad
\begin{cases}
\boldsymbol{s}^*_{\mathsf{a},\mathsf{b}}
=(a^*+L^*b^*,b-La)\\
\eta_{\mathsf{a},\mathsf{b}}
=\scal{a}{a^*}+\scal{b}{b^*}.
\end{cases}
\end{equation}
Then the following hold:
\begin{enumerate}
\item
\label{77htFq10i}
Let $\mathsf{a}\in\gra A$ and $\mathsf{b}\in\gra B$. Then
$\boldsymbol{s}^*_{\mathsf{a},\mathsf{b}}
=\boldsymbol{0}\;\Leftrightarrow\;
\boldsymbol{H}_{\mathsf{a},\mathsf{b}}
=\KKK\;\Rightarrow\;(a,b^*)\in\boldsymbol{Z}
\;\text{and}\;\eta_{\mathsf{a},\mathsf{b}}=0$.
\item
\label{77htFq10ii}
Let $\mathsf{a}\in\gra A$ and $\mathsf{b}\in\gra B$. Then
$\boldsymbol{Z}\subset\boldsymbol{H}_{\mathsf{a},
\mathsf{b}}$\,.
\item
\label{77htFq10iii}
$\boldsymbol{Z}=\bigcap_{\mathsf{a}\in\gra A}
\bigcap_{\mathsf{b}\in\gra B}
\boldsymbol{H}_{\mathsf{a},\mathsf{b}}$\,.
\item
\label{77htFq10iv}
Let $(a,a^*)\in\gra A$, $(b,b^*)\in\gra B$, and 
$(x,v^*)\in\KKK$. Set $s^*=a^*+L^*b^*$, $t=b-La$, and
$\sigma=\sqrt{\|s^*\|^2+\|t\|^2}$; if $\sigma>0$, set
$\Delta=(\scal{x}{s^*}+\scal{t}{v^*}-\scal{a}{a^*}-\scal{b}{b^*})
/\sigma$. Then
\begin{equation}
\label{kL8u8626}
P_{\boldsymbol{H}_{\mathsf{a},\mathsf{b}}}(x,v^*)=
\begin{cases}
\big(x-(\Delta/\sigma)s^*,v^*-(\Delta/\sigma)t\big),
&\text{if}\;\;\sigma>0\;\;\text{and}\;\;\Delta>0;\\
(x,v^*),&\text{otherwise.}
\end{cases}
\end{equation}
\end{enumerate}
\end{proposition}
\begin{proof}
\ref{77htFq10i}: Suppose that 
$\boldsymbol{s}^*_{\mathsf{a},\mathsf{b}}=\boldsymbol{0}$. 
Then $-L^*b^*=a^*\in Aa$ and $La=b\in B^{-1}b^*$. Hence,
\eqref{aqYg6521a} implies that $(a,b^*)\in\boldsymbol{Z}$. In 
addition, 
\begin{equation}
\eta_{\mathsf{a},\mathsf{b}}=\scal{a}{a^*}
+\scal{b}{b^*}=\scal{a}{-L^*b^*}+\scal{La}{b^*}=
-\scal{La}{b^*}+\scal{La}{b^*}=0
\end{equation}
and therefore
$\boldsymbol{H}_{\mathsf{a},\mathsf{b}}=\KKK$. Conversely, 
$\boldsymbol{H}_{\mathsf{a},\mathsf{b}}=\KKK$ $\Rightarrow$ 
$\boldsymbol{s}^*_{\mathsf{a},\mathsf{b}}=\boldsymbol{0}$
and $\eta_{\mathsf{a},\mathsf{b}}=0$.

\ref{77htFq10ii}:
Suppose that $\boldsymbol{x}=(x,v^*)\in\boldsymbol{Z}$.
Then $(x,-L^*v^*)\in\gra A$ and, by monotonicity of $A$,
\begin{equation}
\label{aqYg6521g}
\scal{a-x}{a^*+L^*v^*}\geq 0.
\end{equation}
Likewise, since $(Lx,v^*)\in\gra B$, we have  
\begin{equation}
\label{aqYg6521h}
\scal{b-Lx}{b^*-v^*}\geq 0.
\end{equation}
Using \eqref{aqYg6521g} and \eqref{aqYg6521h}, we obtain
\begin{align}
\label{aqYg6521i}
\sscal{\boldsymbol{x}}{\boldsymbol{s}^*_{\mathsf{a},\mathsf{b}}}
&=\scal{x}{a^*+L^*b^*}+\scal{b-La}{v^*}\nonumber\\
&=\scal{x}{a^*+L^*v^*}+\scal{Lx}{b^*-v^*}
+\scal{b-Lx}{v^*}+\scal{x-a}{L^*v^*}\nonumber\\
&=\scal{x-a}{a^*+L^*v^*}+\scal{a}{a^*}+\scal{La}{v^*}\nonumber\\
&\quad\;+\scal{Lx-b}{b^*-v^*}+\scal{b}{b^*}-\scal{b}{v^*}
+\scal{b-Lx}{v^*}+\scal{x-a}{L^*v^*}\nonumber\\
&\leq\scal{a}{a^*}+\scal{La-b}{v^*}+\scal{b}{b^*}
+\scal{b-Lx}{v^*}+\scal{x-a}{L^*v^*}\nonumber\\
&=\scal{a}{a^*}+\scal{b}{b^*}\nonumber\\
&=\eta_{\mathsf{a},\mathsf{b}}.
\end{align}
Thus, $\boldsymbol{x}\in\boldsymbol{H}_{\mathsf{a},\mathsf{b}}\,$.

\ref{77htFq10iii}:
By \ref{77htFq10ii}, 
$\boldsymbol{Z}\subset\bigcap_{\mathsf{a}\in\gra A}
\bigcap_{\mathsf{b}\in\gra B}
\boldsymbol{H}_{\mathsf{a},\mathsf{b}}$. Conversely, 
fix $\mathsf{a}\in\gra A$ and $\mathsf{b}\in\gra B$, and let 
$\boldsymbol{x}=(x,v^*)\in\boldsymbol{H}_{\mathsf{a},
\mathsf{b}}$.
Then $\sscal{\boldsymbol{x}}{\boldsymbol{s}^*_{\mathsf{a},
\mathsf{b}}}\leq\eta_{\mathsf{a},\mathsf{b}}$ 
and therefore
\begin{align}
\label{kL8u8611a}
\scal{(a,b^*)-(x,v^*)}{(a^*,b)-(-L^*v^*,Lx)}
&=\scal{(a-x,b^*-v^*)}{(a^*+L^*v^*,b-Lx)}\nonumber\\
&=\scal{a-x}{a^*+L^*v^*}+\scal{b-Lx}{b^*-v^*}\nonumber\\
&=\eta_{\mathsf{a},\mathsf{b}}-\scal{\boldsymbol{x}}
{\boldsymbol{s}^*_{\mathsf{a},\mathsf{b}}}\nonumber\\
&\geq 0.
\end{align}
Now set $\boldsymbol{M}\colon\KKK\to 2^{\KKK}\colon (z,w^*)\mapsto 
Az\times B^{-1}w^*$. Then, since $((a,b^*),(a^*,b))$ is an 
arbitrary point in $\gra\boldsymbol{M}$ and since 
\cite[Propositions~20.22 and 20.23]{Livre1} imply that
$\boldsymbol{M}$ is maximally monotone, we derive
from \eqref{kL8u8611a} that
$((x,v^*),(-L^*v^*,Lx))\in\gra\boldsymbol{M}$, i.e., 
that $\boldsymbol{x}\in\boldsymbol{Z}$.

\ref{77htFq10iv}: Let $\boldsymbol{x}\in\KKK$. As seen in 
\ref{77htFq10i}, if 
$\boldsymbol{s}_{\mathsf{a},\mathsf{b}}^*=\boldsymbol{0}$, then 
$\eta_{\mathsf{a},\mathsf{b}}=0$ and 
$\boldsymbol{H}_{\mathsf{a},\mathsf{b}}=\KKK$. Hence
$\sscal{\boldsymbol{x}}{\boldsymbol{s}_{\mathsf{a},\mathsf{b}}^*}=
\eta_{\mathsf{a},\mathsf{b}}$ and 
$P_{\boldsymbol{H}_{\mathsf{a},\mathsf{b}}}\boldsymbol{x}
=\boldsymbol{x}$. Otherwise, it follows from 
\cite[Example~28.16]{Livre1} that
\begin{equation}
\label{kL8u8627}
P_{\boldsymbol{H}_{\mathsf{a},\mathsf{b}}}\boldsymbol{x}=
\begin{cases}
\boldsymbol{x}-\Frac{\sscal{\boldsymbol{x}}
{\boldsymbol{s}^*_{\mathsf{a},\mathsf{b}}}-
\eta_{\mathsf{a},\mathsf{b}}}
{\|\boldsymbol{s}_{\mathsf{a},\mathsf{b}}^*\|^2}
\boldsymbol{s}_{\mathsf{a},\mathsf{b}}^*\,,&
\text{if}\;\;\sscal{\boldsymbol{x}}
{\boldsymbol{s}_{\mathsf{a},\mathsf{b}}^*}>
\eta_{\mathsf{a},\mathsf{b}}\,;\\
\boldsymbol{x},&\text{otherwise.}
\end{cases}
\end{equation}
In view of \eqref{kL8u8625}, the proof is complete.
\end{proof}

\begin{remark}
\label{8hChtFq08}\
\rm
\begin{enumerate}
\item
\label{8hChtFq08i}
The fact that $\boldsymbol{Z}$ is closed and convex 
(Proposition~\ref{p:uuYt6y31z}\ref{p:uuYt6y31zi}) 
is also apparent in 
Proposition~\ref{77htFq10}\ref{77htFq10iii}, which 
exhibits $\boldsymbol{Z}$ as an intersection of closed affine
half-spaces.
\item
\label{8hChtFq08ii}
The inclusion 
$\boldsymbol{Z}\subset\boldsymbol{H}_{\mathsf{a},\mathsf{b}}$ 
(Proposition~\ref{77htFq10}\ref{77htFq10i}) will play a 
key role in the paper. This construction is inspired by that of
\cite[Lemma~3]{Svai08}, where $\GG=\HH$ and $L=\Id$.
\end{enumerate}
\end{remark}

Our analysis will require the following asymptotic principle, 
which is of interest in its own right. 

\begin{proposition}
\label{H7gg5slbu515}
In the setting of Problem~\ref{prob:1}, 
let $(a_n,a_n^*)_{n\in\NN}$ be a sequence in $\gra A$, 
let $(b_n,b_n^*)_{n\in\NN}$ be a sequence in $\gra B$, 
and let $(\overline{x},\overline{v}^*)\in\KKK$. 
Suppose that $a_n\weakly\overline{x}$, $b^*_n\weakly\overline{v}^*$, 
$a^*_n+L^*b^*_n\to 0$, and $La_n-b_n\to 0$.
Then $\scal{a_n}{a_n^*}+\scal{b_n}{b_n^*}\to 0$ 
and $(\overline{x},\overline{v}^*)\in\boldsymbol{Z}$. 
\end{proposition}
\begin{proof}
Define
\begin{equation}
\label{eyoi8Yt533815a}
\boldsymbol{V}=\menge{(x,y)\in\KKK}{Lx=y}.
\end{equation}
Then
\begin{equation}
\label{eyoi8Yt533815b}
\boldsymbol{V}^\bot=\menge{(u^*,v^*)\in\KKK}{u^*=-L^*v^*}.
\end{equation}
Now set
\begin{equation}
\label{eyoi8Yt533815d}
\boldsymbol{A}\colon\KKK\to 2^{\KKK}\colon(x,y)\mapsto Ax\times By.
\end{equation}
We deduce from \eqref{aqYg6521a} that, for every 
$(x,v^*)\in\KKK$,
\begin{eqnarray}
\label{eyoi8Yt533815r}
(x,v^*)\in\boldsymbol{Z}
&\Leftrightarrow&
\begin{cases}
(x,-L^*v^*)\in\gra A\\
(Lx,v^*)\in\gra B
\end{cases}
\nonumber\\
&\Leftrightarrow&
(\boldsymbol{x},\boldsymbol{u})=
\big((x,Lx),(-L^*v^*,v^*)\big)
\in(\boldsymbol{V}\times
\boldsymbol{V}^\bot)\cap\gra\boldsymbol{A}.
\end{eqnarray}
On the other hand, \cite[Lemma~3.1]{Optl14} asserts that
\begin{align}
\label{eyoi8Yt533815s}
(\forall (x,y)\in\KKK)\quad 
\begin{cases}
P_{\boldsymbol{V}}(x,y)
=\big((\Id+L^*L)^{-1}(x+L^*y),L(\Id+L^*L)^{-1}(x+L^*y)\big)\\
P_{\boldsymbol{V}^\bot}(x,y)
=\big(L^*(\Id+LL^*)^{-1}(Lx-y),-(\Id+LL^*)^{-1}(Lx-y)\big).
\end{cases}
\end{align}
Now set
\begin{equation}
\label{eyoi8Yt533816e}
\boldsymbol{\overline{x}}=(\overline{x},L\overline{x}),\quad 
\boldsymbol{\overline{u}}=(-L^*\overline{v}^*,\overline{v}^*),
\quad\text{and}\quad
(\forall n\in\NN)\quad
\begin{cases}
\boldsymbol{x}_n=(a_n,b_n)\\
\boldsymbol{u}_n=(a^*_n,b_n^*).
\end{cases}
\end{equation}
Since $a^*_n+L^*b^*_n\to 0$ and $La_n-b_n\to 0$, we derive from
\eqref{eyoi8Yt533815s} that 
$P_{\boldsymbol{V}}\boldsymbol{u}_n\to\boldsymbol{0}$
and $P_{\boldsymbol{V}^\bot}\boldsymbol{x}_n\to\boldsymbol{0}$.
Altogether, since $L$ and $L^*$ are weakly continuous, the 
assumptions yield
\begin{equation}
\label{eyoi8Yt533815e}
\begin{cases}
(\forall n\in\NN)\;\:(\boldsymbol{x}_n,\boldsymbol{u}_n)
\in\gra\boldsymbol{A}\\
\boldsymbol{x}_n\weakly
\boldsymbol{\overline{x}}\\
\boldsymbol{u}_n\weakly\boldsymbol{\overline{u}}\\
P_{\boldsymbol{V}^\bot}\boldsymbol{x}_n\to\boldsymbol{0}\\
P_{\boldsymbol{V}}\boldsymbol{u}_n\to\boldsymbol{0}.
\end{cases}
\end{equation}
However, \eqref{eyoi8Yt533815e} and 
\cite[Proposition~25.3]{Livre1} imply that
\begin{equation}
\label{eyoi8Yt533815m}
\scal{\boldsymbol{x}_n}{\boldsymbol{u}_n}\to 0
\quad\text{and}\quad
(\boldsymbol{\overline{x}},\boldsymbol{\overline{u}})
\in(\boldsymbol{V}\times\boldsymbol{V}^\bot)\cap\gra\boldsymbol{A}.
\end{equation}
In view of \eqref{eyoi8Yt533815r}, the proof is complete. 
\end{proof}

\begin{remark}
\rm
In the special case when $\GG=\HH$ and $L=\Id$, 
Proposition~\ref{H7gg5slbu515} reduces to
\cite[Corollary~3]{Baus09} (see also \cite[Corollary~25.5]{Livre1}
for an alternate proof), where $m=2$.
The decomposition $\KKK=\boldsymbol{V}\oplus\boldsymbol{V}^\bot$, 
where $\boldsymbol{V}$ is as in \eqref{eyoi8Yt533815a}, is used 
in \cite{Optl14} in a different context.
\end{remark}

\section{Finding Kuhn-Tucker points by Fej\'er approximations}
\label{sec:3}

In view of
Proposition~\ref{p:uuYt6y31z}\ref{p:uuYt6y31zi}, 
Problem~\ref{prob:1} reduces to finding a point in a nonempty 
closed convex subset of a Hilbert space. This can be achieved 
via the following generic Fej\'er-monotone algorithm.

\begin{proposition}{\rm\cite{Eoop01}}
\label{p:1999}
Let $\HHH$ be a real Hilbert space, let $\boldsymbol{C}$ be a
nonempty closed convex subset of $\HHH$, and let 
$\boldsymbol{x}_0\in\HHH$. Iterate 
\begin{equation}
\label{aqYg6520e}
\begin{array}{l}
\text{for}\;n=0,1,\ldots\\
\left\lfloor
\begin{array}{l}
\text{$\boldsymbol{H}_n$ is a closed affine half-space such that 
$\boldsymbol{C}\subset\boldsymbol{H}_n$}\\
\lambda_n\in\left]0,2\right[\\
\boldsymbol{x}_{n+1}=\boldsymbol{x}_n+
\lambda_n(P_{\boldsymbol{H}_n}\boldsymbol{x}_n-\boldsymbol{x}_n).
\end{array}
\right.\\
\end{array}
\end{equation}
Then the following hold:
\begin{enumerate}
\item
$(\boldsymbol{x}_n)_{n\in\NN}$ is Fej\'er monotone with 
respect to $\boldsymbol{C}$:
$(\forall\boldsymbol{z}\in\boldsymbol{C})(\forall n\in\NN)$
$\|\boldsymbol{x}_{n+1}-\boldsymbol{z}\|\leq
\|\boldsymbol{x}_n-\boldsymbol{z}\|$.
\item
$\sum_{n\in\NN}\lambda_n(2-\lambda_n)\|P_{\boldsymbol{H}_n}
\boldsymbol{x}_{n}-\boldsymbol{x}_n\|^2<\pinf$.
\item
Suppose that, for every $\boldsymbol{x}\in\HHH$ and every 
strictly increasing sequence $(k_n)_{n\in\NN}$ in $\NN$, 
$\boldsymbol{x}_{k_n}\weakly\boldsymbol{x}$
$\Rightarrow$ $\boldsymbol{x}\in\boldsymbol{C}$.
Then $(\boldsymbol{x}_n)_{n\in\NN}$ converges weakly to a point in
$\boldsymbol{C}$.
\end{enumerate}
\end{proposition}

We now derive from the above convergence principle a conceptual 
primal-dual splitting framework.

\begin{proposition}
\label{jkj66ybnCxZ-22}
Consider the setting of Problem~\ref{prob:1}.
Suppose that $\mathscr{P}\neq\emp$, let $x_0\in\HH$, let 
$v_0^*\in\GG$, and iterate
\begin{equation}
\label{aqYg6522e}
\begin{array}{l}
\text{for}\;n=0,1,\ldots\\
\left\lfloor
\begin{array}{l}
(a_n,a_n^*)\in\gra A\\
(b_n,b_n^*)\in\gra B\\
s^*_n=a^*_{n}+L^*b^*_{n}\\
t_n=b_{n}-La_{n}\\
\sigma_n=\sqrt{\|s_n^*\|^2+\|t_n\|^2}\\
\text{if}\;\sigma_n=0\\
\left\lfloor
\begin{array}{l}
\overline{x}=a_n\\
\overline{v}^*=b^*_n\\
\text{terminate}.
\end{array}
\right.\\
\text{if}\;\sigma_n>0\\
\left\lfloor
\begin{array}{l}
\lambda_n\in\left]0,2\right[\\
\Delta_n=\text{\rm max}\big\{0,(\scal{x_n}{s^*_n}+\scal{t_n}{v^*_n}
-\scal{a_n}{a^*_n}-\scal{b_n}{b^*_n})/\sigma_n\big\}\\[2mm]
\theta_n=\lambda_n\Delta_n/\sigma_n\\
x_{n+1}=x_n-\theta_n s^*_n\\
v^*_{n+1}=v^*_n-\theta_n t_n.
\end{array}
\right.\\[8mm]
\end{array}
\right.\\[4mm]
\end{array}
\end{equation}
Then either \eqref{aqYg6522e} terminates at a solution 
$(\overline{x},\overline{v}^*)\in\boldsymbol{Z}$ in a finite 
number of iterations or it generates infinite sequences 
$(x_n)_{n\in\NN}$ and $(v_n^*)_{n\in\NN}$ such that the following
hold:
\begin{enumerate}
\item
\label{jkj66ybnCxZ-22i}
$(x_n,v_n^*)_{n\in\NN}$ is Fej\'er monotone with respect to
$\boldsymbol{Z}$.
\item
\label{jkj66ybnCxZ-22ii}
$\sum_{n\in\NN}\lambda_n(2-\lambda_n)\Delta_n^2<\pinf$.
\item
\label{jkj66ybnCxZ-22iii}
Suppose that for every $x\in\HH$, every $v^*\in\GG$, and every 
strictly increasing sequence $(k_n)_{n\in\NN}$ in $\NN$, 
\begin{equation}
\label{kL8u8627a}
\big[\:x_{k_n}\weakly x\;\;\text{and}\;\;v^*_{k_n}\weakly v^*\:\big]
\quad\Rightarrow\quad (x,v^*)\in\boldsymbol{Z}.
\end{equation}
Then $(x_n)_{n\in\NN}$ converges weakly to a point 
$\overline{x}\in\mathscr{P}$, $(v_n^*)_{n\in\NN}$ converges 
weakly to a point $\overline{v}^*\in\mathscr{D}$, and 
$(\overline{x},\overline{v}^*)\in\boldsymbol{Z}$.
\end{enumerate}
\end{proposition}
\begin{proof}
We first observe that, by Proposition~\ref{p:uuYt6y31z},
$\boldsymbol{Z}$ is nonempty, closed, and convex. 
Two alternatives are possible. 
First, suppose that, for some $n\in\NN$, $\sigma_n=0$. 
Then Proposition~\ref{77htFq10}\ref{77htFq10i} asserts that 
the algorithm terminates at 
$(\overline{x},\overline{v}^*)=(a_n,b_n^*)\in\boldsymbol{Z}$.
Now suppose that $(\forall n\in\NN)$ $\sigma_n>0$.
For every $n\in\NN$, set 
\begin{equation}
\boldsymbol{x}_n=(x_n,v_n^*),\;\;
\boldsymbol{s}_n^*=(s_n^*,t_n),
\quad\text{and}\quad
\eta_n=\scal{a_n}{a^*_n}+\scal{b_n}{b^*_n}, 
\end{equation}
and define
\begin{equation}
\boldsymbol{H}_n=\menge{\boldsymbol{x}\in\KKK}
{\scal{\boldsymbol{x}}{\boldsymbol{s}_n^*}\leq\eta_n}. 
\end{equation}
Then we derive from \eqref{aqYg6522e} and
Proposition~\ref{77htFq10}\ref{77htFq10ii} that 
$(\forall n\in\NN)$ $\boldsymbol{Z}\subset\boldsymbol{H}_n$. On 
the other hand, Proposition~\ref{77htFq10}\ref{77htFq10iv} 
implies that
\begin{equation}
\label{kL8u8629a}
(\forall n\in\NN)\quad\Delta_n=
\|P_{\boldsymbol{H}_n}\boldsymbol{x}_{n}-\boldsymbol{x}_n\|
\quad\text{and}\quad\boldsymbol{x}_{n+1}=\boldsymbol{x}_n+
\lambda_n(P_{\boldsymbol{H}_n}\boldsymbol{x}_n-\boldsymbol{x}_n).
\end{equation}
Thus, the conclusions follow from 
Proposition~\ref{p:uuYt6y31z}\ref{p:uuYt6y31zi}
and Proposition~\ref{p:1999}.
\end{proof}

At the $n$th iteration of algorithm \eqref{aqYg6522e}, one picks
the quadruple $(a_n,a_n^*,b_n,b_n^*)$ in $\gra A\times\gra B$.
In the following corollary, this quadruple is taken in a more 
restricted set adapted to the current primal-dual iterate 
$(x_n,v_n^*)$, which leads to more explicit convergence conditions. 

\begin{corollary}
\label{ccnw7Gjh1106}
Consider the setting of Problem~\ref{prob:1}.
Suppose that $\mathscr{P}\neq\emp$, let $\varepsilon\in\zeroun$, 
let $\alpha\in\RPP$, let $x_0\in\HH$, and let $v_0^*\in\GG$.
For every $(x,v^*)\in\KKK$, set
\begin{multline}
\label{ecnw7Gjh1105a}
\boldsymbol{G}_\alpha(x,v^*)=
\Big\{(a,b,a^*,b^*)\in\KKK\times\KKK\;\big |\;
(a,a^*)\in\gra A,\;(b,b^*)\in\gra B,\;\text{and}\\
\scal{x-a}{a^*+L^*v^*}+\scal{Lx-b}{b^*-v^*}
\geq\alpha\big(\|a^*+L^*b^*\|^2+\|La-b\|^2\big)\Big\}.
\end{multline}
Iterate
\begin{equation}
\label{ecnw7Gjh1105b}
\begin{array}{l}
\text{for}\;n=0,1,\ldots\\
\left\lfloor
\begin{array}{l}
(a_n,b_n,a_n^*,b_n^*)\in\boldsymbol{G}_\alpha(x_n,v_n^*)\\
s^*_n=a^*_n+L^*b^*_n\\
t_n=b_n-La_n\\
\tau_n=\|s_n^*\|^2+\|t_n\|^2\\
\text{if}\;\tau_n=0\\
\left\lfloor
\begin{array}{l}
\overline{x}=a_n\\
\overline{v}^*=b^*_n\\
\text{terminate}.
\end{array}
\right.\\
\text{if}\;\tau_n>0\\
\left\lfloor
\begin{array}{l}
\lambda_n\in\left[\varepsilon,2-\varepsilon\right]\\
\theta_n=\lambda_n\big(\scal{x_n}{s^*_n}+\scal{t_n}{v^*_n}
-\scal{a_n}{a^*_n}-\scal{b_n}{b^*_n}\big)/\tau_n\\
x_{n+1}=x_n-\theta_n s^*_n\\
v^*_{n+1}=v^*_n-\theta_n t_n.
\end{array}
\right.\\[8mm]
\end{array}
\right.\\[4mm]
\end{array}
\end{equation}
Then either \eqref{ecnw7Gjh1105b} terminates at a solution 
$(\overline{x},\overline{v}^*)\in\boldsymbol{Z}$ in a finite 
number of iterations or it generates infinite sequences 
$(x_n)_{n\in\NN}$ and $(v_n^*)_{n\in\NN}$ such that the following
hold:
\begin{enumerate}
\item
\label{ccnw7Gjh1106i}
$\sum_{n\in\NN}\|s^*_n\|^2<\pinf$ and 
$\sum_{n\in\NN}\|t_n\|^2<\pinf$.
\item
\label{ccnw7Gjh1106ii}
$\sum_{n\in\NN}\|x_{n+1}-x_n\|^2<\pinf$ and 
$\sum_{n\in\NN}\|v^*_{n+1}-v^*_n\|^2<\pinf$.
\item
\label{ccnw7Gjh1106iii}
Suppose that 
\begin{equation}
\label{ecnw7Gjh1106iiia}
x_n-a_n\weakly 0\quad\text{and}\quad v_n^*-b_n^*\weakly 0.
\end{equation}
Then $(x_n)_{n\in\NN}$ converges weakly to a point 
$\overline{x}\in\mathscr{P}$, $(v_n^*)_{n\in\NN}$ converges 
weakly to a point $\overline{v}^*\in\mathscr{D}$, and 
$(\overline{x},\overline{v}^*)\in\boldsymbol{Z}$.
\end{enumerate}
\end{corollary}
\begin{proof}
This corollary is an application of Proposition~\ref{jkj66ybnCxZ-22}.
To see this, let $(x,v^*)\in\KKK$. First, to show that the 
algorithm is well defined, we must prove that 
$\boldsymbol{G}_\alpha(x,v^*)\neq\emp$. 
Since $\mathscr{P}\neq\emp$, it follows from
Proposition~\ref{p:uuYt6y31z}\ref{p:uuYt6y31zii} that
$\boldsymbol{Z}\neq\emp$. Now let $(a,b^*)\in\boldsymbol{Z}$,
and set $a^*=-L^*b^*$ and $b=La$. Then \eqref{aqYg6521a} yields 
$(a,a^*)\in\gra A$ and $(b,b^*)\in\gra B$. Moreover,
\begin{align}
\label{kL8u8606a}
&\hskip -16mm
\scal{x-a}{a^*+L^*v^*}+\scal{Lx-b}{b^*-v^*}\nonumber\\
&=-\scal{x-a}{L^*(b^*-v^*)}+\scal{L(x-a)}{b^*-v^*}\nonumber\\
&=0\nonumber\\
&=\alpha\big(\|a^*+L^*b^*\|^2+\|La-b\|^2\big).
\end{align}
Hence $(a,b,a^*,b^*)\in\boldsymbol{G}_\alpha(x,v^*)$ and
\eqref{ecnw7Gjh1105b} is well-defined. Next, to show that 
\eqref{ecnw7Gjh1105b} is a special case of \eqref{aqYg6522e}
it is enough to consider the case when $(\forall n\in\NN)$ 
$\tau_n>0$. Note that \eqref{ecnw7Gjh1105b} yields
\begin{align}
\label{kL8u8606b}
&(\forall n\in\NN)\quad\scal{x_n}{s^*_n}+\scal{t_n}{v^*_n}
-\scal{a_n}{a^*_n}-\scal{b_n}{b^*_n}\nonumber\\
&\hskip 32mm =\scal{x_n-a_n}{a_n^*+L^*v_n^*}
+\scal{Lx_n-b_n}{b_n^*-v_n^*}\nonumber\\
&\hskip 32mm \geq\alpha\big(\|a_n^*+L^*b_n^*\|^2+\|La_n-b_n\|^2\big)
\nonumber\\
&\hskip 32mm =\alpha\tau_n\nonumber\\
&\hskip 32mm >0.
\end{align}
In turn, if we define $(\Delta_n)_{n\in\NN}$ as in 
\eqref{aqYg6522e}, we obtain
\begin{equation}
\label{kL8u8606c}
(\forall n\in\NN)\quad\Delta_n=
\frac{\scal{x_n}{s^*_n}+\scal{t_n}{v^*_n}
-\scal{a_n}{a^*_n}-\scal{b_n}{b^*_n}}{\sqrt{\tau_n}}
\geq\alpha\sqrt{\tau_n}>0.
\end{equation}
Hence \eqref{ecnw7Gjh1105b} is a special case of 
\eqref{aqYg6522e}. Moreover, it follows from 
\eqref{kL8u8606c} and 
Proposition~\ref{jkj66ybnCxZ-22}\ref{jkj66ybnCxZ-22ii} that
\begin{equation}
\label{e:france-musique}
\sum_{n\in\NN}\big(\|s_n^*\|^2+\|t_n\|^2\big)
=\sum_{n\in\NN}\tau_n
\leq\frac{1}{\alpha^2}\sum_{n\in\NN}\Delta_n^2
\leq\frac{1}{(\alpha\varepsilon)^2}
\sum_{n\in\NN}\lambda_n(2-\lambda_n)\Delta_n^2<\pinf,
\end{equation}
which establishes \ref{ccnw7Gjh1106i}. 
On the other hand, \ref{ccnw7Gjh1106ii} 
results from \eqref{ecnw7Gjh1105b} and 
\eqref{e:france-musique} since
\begin{align}
\label{ecnw7Gjh1112g}
\sum_{n\in\NN}\big(\|x_{n+1}-x_n\|^2+\|v^*_{n+1}-v^*_n\|^2\big)
&=\sum_{n\in\NN}\theta_n^2\tau_n
\nonumber\\
&=\sum_{n\in\NN}\lambda_n^2\Delta_n^2
\nonumber\\
&\leq(2-\varepsilon)^2\sum_{n\in\NN}\Delta_n^2\nonumber\\
&<\pinf.
\end{align}
Finally, to prove \ref{ccnw7Gjh1106iii}, it remains to check
\eqref{kL8u8627a}. Take $x\in\HH$, $v^*\in\GG$, and a
strictly increasing sequence $(k_n)_{n\in\NN}$ in $\NN$ 
such that $x_{k_n}\weakly x$ and $v^*_{k_n}\weakly v^*$.
Then it follows from \eqref{ecnw7Gjh1106iiia} and 
\ref{ccnw7Gjh1106i} that 
\begin{equation}
\label{PkL87hT608k}
a_{k_n}\weakly{x},\quad
b^*_{k_n}\weakly{v^*},\quad
a^*_{k_n}+L^*b^*_{k_n}\to 0,\quad\text{and}\quad
La_{k_n}-b_{k_n}\to 0,
\end{equation}
and from \eqref{ecnw7Gjh1105b} that $(\forall n\in\NN)$ 
$(a_n,a_n^*)\in\gra A$ and $(b_n,b_n^*)\in\gra B$. We therefore 
appeal to Proposition~\ref{H7gg5slbu515} to conclude that 
$(x,v^*)\in\boldsymbol{Z}$.
\end{proof}

\begin{remark}
\label{rcnw7Gjh1112}
\rm
In the special case when $\GG=\HH$ and $L=\Id$, 
Corollary~\ref{ccnw7Gjh1106}\ref{ccnw7Gjh1106iii} was established 
in \cite[Proposition~2]{Svai08} under the following additional
assumptions: $A+B$ is maximally monotone or $\HH$ is
finite-dimensional, $x_n-a_n\to 0$, and $v_n^*-b_n^*\to 0$.
\end{remark}

Corollary~\ref{ccnw7Gjh1106} is conceptual in that it does not
specify a rule for selecting the quadruple 
$(a_n,b_n,a_n^*,b_n^*)$ in $\boldsymbol{G}_\alpha(x_n,v_n^*)$
at iteration $n$. We now provide an example of a concrete 
selection rule. 

\begin{proposition}
\label{pmnmMyt^rw-19}
Consider the setting of Problem~\ref{prob:1}.
Suppose that $\mathscr{P}\neq\emp$, let $\varepsilon\in\zeroun$, 
let $x_0\in\HH$, let $v_0^*\in\GG$, and iterate
\begin{equation}
\label{emnmMyt^rw-19d}
\begin{array}{l}
\text{for}\;n=0,1,\ldots\\
\left\lfloor
\begin{array}{l}
(\gamma_n,\mu_n)\in [\varepsilon,1/\varepsilon]^2\\
a_n=J_{\gamma_n A}(x_n-\gamma_n L^*v_n^*)\\
l_n=Lx_n\\
b_n=J_{\mu_n B}(l_n+\mu_n v_n^*)\\
s^*_n=\gamma_n^{-1}(x_n-a_n)+\mu_n^{-1}L^*(l_n-b_n)\\
t_n=b_{n}-La_{n}\\
\tau_n=\|s_n^*\|^2+\|t_n\|^2\\
\text{if}\;\tau_n=0\\
\left\lfloor
\begin{array}{l}
\overline{x}=a_n\\
\overline{v}^*=v_n^*+\mu_n^{-1}(l_n-b_n)\\
\text{terminate}.
\end{array}
\right.\\
\text{if}\;\tau_n>0\\
\left\lfloor
\begin{array}{l}
\lambda_n\in\left[\varepsilon,2-\varepsilon\right]\\
\theta_n=\lambda_n\big(\gamma_n^{-1}\|x_n-a_n\|^2+\mu_n^{-1}
\|l_n-b_n\|^2\big)/\tau_n\\
x_{n+1}=x_n-\theta_n s^*_n\\
v^*_{n+1}=v^*_n-\theta_n t_n.
\end{array}
\right.\\[8mm]
\end{array}
\right.\\[4mm]
\end{array}
\end{equation}
Then either \eqref{emnmMyt^rw-19d} terminates at a solution 
$(\overline{x},\overline{v}^*)\in\boldsymbol{Z}$ in a finite 
number of iterations or it generates infinite sequences 
$(x_n)_{n\in\NN}$ and $(v_n^*)_{n\in\NN}$
such that the following hold:
\begin{enumerate}
\item
\label{pmnmMyt^rw-19i}
$\sum_{n\in\NN}\|s^*_n\|^2<\pinf$ and 
$\sum_{n\in\NN}\|t_n\|^2<\pinf$.
\item
\label{pmnmMyt^rw-19ii}
$\sum_{n\in\NN}\|x_{n+1}-x_n\|^2<\pinf$ and 
$\sum_{n\in\NN}\|v^*_{n+1}-v^*_n\|^2<\pinf$.
\item
\label{pmnmMyt^rw-19ii'}
$\sum_{n\in\NN}\|x_n-a_n\|^2<\pinf$ and 
$\sum_{n\in\NN}\|Lx_n-b_n\|^2<\pinf$.
\item
\label{pmnmMyt^rw-19iii}
$(x_n)_{n\in\NN}$ converges weakly to a point 
$\overline{x}\in\mathscr{P}$, $(v_n^*)_{n\in\NN}$ converges 
weakly to a point $\overline{v}^*\in\mathscr{D}$, and 
$(\overline{x},\overline{v}^*)\in\boldsymbol{Z}$.
\end{enumerate}
\end{proposition}
\begin{proof}
We are going to derive the results from 
Corollary~\ref{ccnw7Gjh1106}. To this end, let us set 
\begin{equation}
\label{emnmMyt^rw-19e}
(\forall n\in\NN)\quad
a^*_n=\gamma_n^{-1}(x_n-a_n)-L^*v_n^*\quad\text{and}\quad
b^*_n=\mu_n^{-1}(Lx_n-b_n)+v_n^*.
\end{equation}
Now let $n\in\NN$. To show that \eqref{emnmMyt^rw-19d} is 
an instantiation of \eqref{ecnw7Gjh1105b}, let us check that 
there exists $\alpha\in\RPP$ such that
$(a_n,b_n,a_n^*,b_n^*)\in\boldsymbol{G}_\alpha(x_n,v_n^*)$.
By construction, we have 
\begin{equation}
\label{emnmMyt^rw-19s}
(a_n,a_n^*)\in\gra A
\quad\text{and}\quad
(b_n,b_n^*)\in\gra B. 
\end{equation}
In view of \eqref{ecnw7Gjh1105a}, we must find $\alpha\in\RPP$ 
such that 
\begin{equation}
\label{emnmMyt^rw-19f}
\scal{x_n-a_n}{a_n^*+L^*v_n^*}+\scal{Lx_n-b_n}{b_n^*-v_n^*}\geq
\alpha\big(\|a^*_n+L^*b^*_n\|^2+\|La_n-b_n\|^2\big).
\end{equation}
By \eqref{emnmMyt^rw-19e}, 
\begin{equation}
\label{emnmMyt^rw-19m}
\scal{x_n-a_n}{a_n^*+L^*v_n^*}+\scal{Lx_n-b_n}{b_n^*-v_n^*}
=\gamma_n^{-1}\|x_n-a_n\|^2+\mu_n^{-1}\|Lx_n-b_n\|^2
\end{equation}
and 
\begin{multline}
\label{emnmMyt^rw-19g}
\|a^*_n+L^*b^*_n\|^2+\|La_n-b_n\|^2\\
=\|(\gamma_n^{-1}\Id+\mu_n^{-1}L^*L)x_n-
(\gamma_n^{-1}a_n+\mu_n^{-1}L^*b_n)\|^2+\|La_n-b_n\|^2.
\end{multline}
On the other hand, 
\begin{equation}
\label{emnmMyt^rw-19h}
\|La_n-b_n\|^2=\|La_n\|^2-2\scal{La_n}{b_n}+\|b_n\|^2
\end{equation}
and
\begin{align}
\label{emnmMyt^rw-19i}
&\hskip -8mm
\|(\gamma_n^{-1}\Id+\mu_n^{-1}L^*L)x_n-
(\gamma_n^{-1}a_n+\mu_n^{-1}L^*b_n)\|^2\nonumber\\
&=\|\gamma_n^{-1}x_n+\mu_n^{-1}L^*Lx_n\|^2-
2\scal{\gamma_n^{-1}x_n+\mu_n^{-1}L^*Lx_n}
{\gamma_n^{-1}a_n+\mu_n^{-1}L^*b_n}
\nonumber\\
&\quad\;+\|\gamma_n^{-1}a_n+\mu_n^{-1}L^*b_n\|^2\nonumber\\
&=\gamma_n^{-2}\|x_n\|^2+2\gamma_n^{-1}\mu_n^{-1}\|Lx_n\|^2+
\mu_n^{-2}\|L^*Lx_n\|^2
-2\gamma_n^{-2}\scal{x_n}{a_n}\nonumber\\
&\quad\;
-2\gamma_n^{-1}\mu_n^{-1}\scal{Lx_n}{b_n}-
2\gamma_n^{-1}\mu_n^{-1}\scal{Lx_n}{La_n}
-2\mu_n^{-2}\scal{L^*Lx_n}{L^*b_n}\nonumber\\
&\quad\;+\gamma_n^{-2}\|a_n\|^2+
2\gamma_n^{-1}\mu_n^{-1}\scal{La_n}{b_n}
+\mu_n^{-2}\|L^*b_n\|^2\nonumber\\
&=\gamma_n^{-2}\|x_n-a_n\|^2+\mu_n^{-2}\|L^*(Lx_n-b_n)\|^2+
2\gamma_n^{-1}\mu_n^{-1}\|Lx_n\|^2\nonumber\\
&\quad\;-2\gamma_n^{-1}\mu_n^{-1}\scal{Lx_n}{b_n}-
2\gamma_n^{-1}\mu_n^{-1}\scal{Lx_n}{La_n}
+2\gamma_n^{-1}\mu_n^{-1}\scal{La_n}{b_n}\nonumber\\
&=\gamma_n^{-2}\|x_n-a_n\|^2+\mu_n^{-2}\|L^*(Lx_n-b_n)\|^2
+\gamma_n^{-1}\mu_n^{-1}\|L(x_n-a_n)\|^2\nonumber\\
&\quad\;-\gamma_n^{-1}\mu_n^{-1}\|La_n\|^2+
\gamma_n^{-1}\mu_n^{-1}\|Lx_n-b_n\|^2-\gamma_n^{-1}\mu_n^{-1}
\|b_n\|^2\nonumber\\
&\quad\;+2\gamma_n^{-1}\mu_n^{-1}\scal{La_n}{b_n}.
\end{align}
Combining \eqref{emnmMyt^rw-19g}, 
\eqref{emnmMyt^rw-19h}, and \eqref{emnmMyt^rw-19i},
and recalling that 
$\{\gamma_n,\mu_n\}\subset[\varepsilon,\varepsilon^{-1}]$, we obtain
\begin{align}
\label{emnmMyt^rw-19j}
&\hskip -8mm
\|a^*_n+L^*b^*_n\|^2+\|La_n-b_n\|^2\nonumber\\
&=\gamma_n^{-2}\|x_n-a_n\|^2+\mu_n^{-2}\|L^*(Lx_n-b_n)\|^2
\nonumber\\
&\quad\;+\gamma_n^{-1}\mu_n^{-1}\|L(x_n-a_n)\|^2
+\gamma_n^{-1}\mu_n^{-1}\|Lx_n-b_n\|^2\nonumber\\
&\quad\;+\big(1-\gamma_n^{-1}\mu_n^{-1}\big)\big(\|La_n\|^2
-2\scal{La_n}{b_n}+\|b_n\|^2\big)\nonumber\\
&=\gamma_n^{-2}\|x_n-a_n\|^2+\mu_n^{-2}\|L^*(Lx_n-b_n)\|^2
+\gamma_n^{-1}\mu_n^{-1}\|L(x_n-a_n)\|^2\nonumber\\
&\quad\;+\gamma_n^{-1}\mu_n^{-1}\|Lx_n-b_n\|^2
+\big(1-\gamma_n^{-1}\mu_n^{-1}\big)\|La_n-b_n\|^2\nonumber\\
&\leq\varepsilon^{-1}\Big(\gamma_n^{-1}\|x_n-a_n\|^2
+\mu_n^{-1}\|L^*(Lx_n-b_n)\|^2
+\gamma_n^{-1}\|L(x_n-a_n)\|^2\nonumber\\
&\quad\;+\mu_n^{-1}\|Lx_n-b_n\|^2\Big)
+2\big(1-\gamma_n^{-1}\mu_n^{-1}\big)\big(\|L(a_n-x_n)\|^2
+\|Lx_n-b_n\|^2\big)\nonumber\\
&\leq\varepsilon^{-1}\big(1+\|L\|^2\big)
\Big(\gamma_n^{-1}\|x_n-a_n\|^2+\mu_n^{-1}\|Lx_n-b_n\|^2\Big)
\nonumber\\
&\quad\;+2\big(\gamma_n-\mu_n^{-1}\big)
\gamma_n^{-1}\|L\|^2\,\|x_n-a_n\|^2+2\big(\mu_n-\gamma_n^{-1}\big)
\mu_n^{-1}\|Lx_n-b_n\|^2\nonumber\\
&\leq\varepsilon^{-1}\big(1+\|L\|^2+2(1-\varepsilon^{2})
\text{max}\big\{1,\|L\|^2\big\}\big)\nonumber\\
&\quad\hskip 44mm\times\big(\gamma_n^{-1}\|x_n-a_n\|^2
+\mu_n^{-1}\|Lx_n-b_n\|^2\big).
\end{align}
Therefore, \eqref{emnmMyt^rw-19m} implies that
\eqref{emnmMyt^rw-19f} is satisfied with 
\begin{equation}
\alpha=\frac{\varepsilon}{1+\|L\|^2+2(1-\varepsilon^{2})
\text{max}\big\{1,\|L\|^2\big\}}.
\end{equation}
We thus obtain \ref{pmnmMyt^rw-19i} and 
\ref{pmnmMyt^rw-19ii}.
To prove \ref{pmnmMyt^rw-19ii'}, note that it follows from
\eqref{emnmMyt^rw-19e} that
\begin{align}
\label{emnmMyt^rw-19t}
x_n-a_n
&=(\gamma_n^{-1}\Id+\mu_n^{-1}L^*L)^{-1}
\big(\gamma_n^{-1}(x_n-a_n)+\mu_n^{-1}L^*L(x_n-a_n)\big)\nonumber\\
&=(\gamma_n^{-1}\Id+\mu_n^{-1}L^*L)^{-1}
\big(\gamma_n^{-1}(x_n-a_n)+\mu_n^{-1}L^*(Lx_n-b_n)\nonumber\\
&\quad\hskip 34mm+\mu_n^{-1}L^*(b_n-La_n)\big)\nonumber\\
&=\gamma_n(\Id+(\gamma_n/\mu_n)L^*L)^{-1}\big((a^*_n+L^*b^*_n)
+\mu_n^{-1}L^*(b_n-La_n)\big).
\end{align}
Thus, since $\|(\Id+(\gamma_n/\mu_n)L^*L)^{-1}\|\leq 1$ and 
since $\text{max}\{\gamma_n,\mu_n^{-1}\}\leq\varepsilon^{-1}$,
we have
\begin{align}
\label{PkL87hT629t}
\|x_n-a_n\|^2
&\leq\gamma_n^2\big\|(\Id+(\gamma_n/\mu_n)L^*L)^{-1}\big\|^2\,
\big(\|a^*_n+L^*b^*_n\|+\mu_n^{-1}\|L^*(La_n-b_n)\|\big)^2
\nonumber\\
&\leq 2\gamma_n^2\big(\|a^*_n+L^*b^*_n\|^2+
\mu_n^{-2}\|L^*(La_n-b_n)\|^2\big)\nonumber\\
&\leq 2\varepsilon^{-2}\big(\|s^*_n\|^2+\varepsilon^{-2}
\|L\|^2\,\|t_n\|^2\big).
\end{align}
Hence, \ref{pmnmMyt^rw-19i} yields 
$\sum_{n\in\NN}\|x_n-a_n\|^2<\pinf$.
In turn, since
\begin{equation}
\label{PkL87hT608g}
\|Lx_n-b_n\|^2=\|L(x_n-a_n)+La_n-b_n\|^2\leq 
2\big(\|L\|^2\,\|x_n-a_n\|^2+\|t_n\|^2\big),
\end{equation}
we obtain $\sum_{n\in\NN}\|Lx_n-b_n\|^2<\pinf$.
Therefore 
\begin{equation}
\sum_{n\in\NN}\|v_n^*-b_n^*\|^2
=\sum_{n\in\NN}\mu_n^{-2}\|Lx_n-b_n\|^2
\leq\varepsilon^{-2}\sum_{n\in\NN}\|Lx_n-b_n\|^2<\pinf.
\end{equation}
Thus, \eqref{ecnw7Gjh1106iiia} is satisfied and 
\ref{pmnmMyt^rw-19iii} follows.
\end{proof}

\begin{remark} 
\label{r:1}
\rm
As mentioned in the Introduction, existing methods for solving 
Problem~\ref{prob:1} either require knowledge of $\|L\|$ or
necessitate potentially hard to implement inversions of linear 
operators. For instance, the method of \cite{Siop11}, which hinges 
on a reformulation that employs Tseng's forward-backward-forward 
algorithm \cite{Tsen00}, imposes the same scaling coefficients on 
$A$ and $B$ at each iteration and they must be bounded by a 
specific constant which depends on $\|L\|$; more precisely, 
$(\forall n\in\NN)$ $\gamma_n=\mu_n\in\left]0,1/\|L\|\right[$. 
These restrictions are lifted in \eqref{emnmMyt^rw-19d},
where the parameters $(\gamma_n)_{n\in\NN}$ and 
$(\mu_n)_{n\in\NN}$ can evolve freely in an
arbitrarily large interval of $\RPP$ independent from $L$. 
\end{remark}

We now highlight two particular instances of interest.

\begin{example}
\label{excnw7Gjh1108}
\rm
Consider the setting of Problem~\ref{prob:1} with $A=0$.
Then the primal problem \eqref{e:primal} reduces to 
\begin{equation}
\label{ecnw7Gjh1109p}
\text{find}\;\;\overline{x}\in\HH\;\;\text{such that}\;\; 
0\in L^*BL\overline{x}.
\end{equation}
Assume that it has at least one solution,
let $\lambda\in\left]0,2\right[$, let $x_0\in\HH$, 
let $v_0^*\in\GG$, and iterate
\begin{equation}
\label{ecnw7Gjh1108}
\begin{array}{l}
\text{for}\;n=0,1,\ldots\\
\left\lfloor
\begin{array}{l}
a_n=x_n-L^*v_n^*\\
l_n=Lx_n\\
b_n=J_{B}(l_n+v_n^*)\\
s^*_n=x_n-a_n+L^*(l_n-b_n)\\
t_n=b_n-La_n\\
\tau_n=\|s_n^*\|^2+\|t_n\|^2\\
\text{if}\;\tau_n=0\\
\left\lfloor
\begin{array}{l}
\overline{x}=a_n\\
\overline{v}^*=v_n^*+l_n-b_n\\
\text{terminate}.
\end{array}
\right.\\
\text{if}\;\tau_n>0\\
\left\lfloor
\begin{array}{l}
\theta_n=\lambda\big(\|x_n-a_n\|^2+\|l_n-b_n\|^2\big)/\tau_n\\
x_{n+1}=x_n-\theta_n s^*_n\\
v^*_{n+1}=v^*_n-\theta_n t_n.
\end{array}
\right.\\[8mm]
\end{array}
\right.\\[4mm]
\end{array}
\end{equation}
This algorithm is the instance of \eqref{emnmMyt^rw-19d} 
in which $A=0$ and $(\forall n\in\NN)$ $\gamma_n=\mu_n=1$ and 
$\lambda_n=\lambda$. It follows from 
Proposition~\ref{pmnmMyt^rw-19} that, if it does not 
terminate, it produces a sequence $(x_n)_{n\in\NN}$ that converges 
weakly to a solution to \eqref{ecnw7Gjh1109p}. 
In the special case when $\HH=\RR^N$ and 
$\GG=\RR^M$ this result was established in \cite{Dong05} (the 
fact that \cite[Algorithm~3.1]{Dong05} is equivalent to 
\eqref{ecnw7Gjh1108} follows from elementary manipulations). 
Interestingly, the analysis
of \cite{Dong05} is quite different from ours and it does not
employ a geometric construction. 
\end{example}

\begin{example}
\label{excnw7Gjh1101}
\rm
Consider the setting of Problem~\ref{prob:1} with $\GG=\HH$
and $L=\Id$. Then the primal problem \eqref{e:primal} reduces to 
\begin{equation}
\label{ecnw7Gjh1109a}
\text{find}\;\;\overline{x}\in\HH\;\;\text{such that}\;\; 
0\in A\overline{x}+B\overline{x}.
\end{equation}
Assume that it has at least one solution, let 
$\varepsilon\in\zeroun$, let $x_0\in\HH$, let $v_0^*\in\HH$, 
and iterate
\begin{equation}
\label{emnmMyt^rw-19d'}
\begin{array}{l}
\text{for}\;n=0,1,\ldots\\
\left\lfloor
\begin{array}{l}
(\gamma_n,\mu_n)\in [\varepsilon,1/\varepsilon]^2\\
a_n=J_{\gamma_n A}(x_n-\gamma_n v_n^*)\\
b_n=J_{\mu_n B}(x_n+\mu_n v_n^*)\\
s^*_n=\gamma_n^{-1}(x_n-a_n)+\mu_n^{-1}(x_n-b_n)\\
t_n=b_n-a_n\\
\tau_n=\|s_n^*\|^2+\|t_n\|^2\\
\text{if}\;\tau_n=0\\
\left\lfloor
\begin{array}{l}
\overline{x}=a_n\\
\overline{v}^*=v_n^*+\mu_n^{-1}(x_n-b_n)\\
\text{terminate}.
\end{array}
\right.\\
\text{if}\;\tau_n>0\\
\left\lfloor
\begin{array}{l}
\lambda_n\in\left[\varepsilon,2-\varepsilon\right]\\
\theta_n=\lambda_n\big(\gamma_n^{-1}\|x_n-a_n\|^2+\mu_n^{-1}
\|x_n-b_n\|^2\big)/\tau_n\\
x_{n+1}=x_n-\theta_n s^*_n\\
v^*_{n+1}=v^*_n-\theta_n t_n.
\end{array}
\right.\\[8mm]
\end{array}
\right.\\[4mm]
\end{array}
\end{equation}
Then Proposition~\ref{pmnmMyt^rw-19} asserts that, if the 
algorithm does not terminate, it produces a sequence 
$(x_n,v^*_n)_{n\in\NN}$ that converges weakly to a 
point $(\overline{x},\overline{v}^*)$ such that 
$-\overline{v}^*\in A\overline{x}$ and
$\overline{v}^*\in B\overline{x}$, so that $\overline{x}$ solves 
\eqref{ecnw7Gjh1109a}. Under the additional assumptions that 
$A+B$ is maximally monotone or that $\HH$ is finite-dimensional,
this result was established in \cite[Proposition~3]{Svai08} for a 
version of \eqref{emnmMyt^rw-19d'} in which an additional 
relaxation parameter is allowed in the definition of $a_n$. 
\end{example}

\section{A Fej\'er monotone algorithm for coupled monotone 
inclusions}
\label{sec:4}

Many complex systems feature interactions between several 
variables can be modeled in terms of equilibria involving 
composite monotone operators. A mathematical formulation of such
problems in duality is the following.

\begin{problem}
\label{prob:2}
\rm
Let $m$ and $K$ be strictly positive integers, let
$(\HH_i)_{1\leq i\leq m}$ and $(\GG_k)_{1\leq k\leq K}$ be real 
Hilbert spaces, and set 
$\KKK=\HH_1\oplus\cdots\HH_m\oplus\GG_1\oplus\cdots\oplus\GG_K$.
For every $i\in\{1,\ldots,m\}$ and $k\in\{1,\ldots,K\}$, let 
$A_i\colon\HH_i\to 2^{\HH_i}$ and $B_k\colon\GG_k\to 2^{\GG_k}$ 
be maximally monotone, let $z_i\in\HH_i$, let $r_k\in\GG_k$, 
and let $L_{ki}\colon\HH_i\to\GG_k$ be linear and bounded. 
Consider the coupled inclusions problem
\begin{multline}
\label{N09Kd424p}
\text{find}\;\;\overline{x}_1\in\HH_1,\ldots,\overline{x}_m\in\HH_m
\;\;\text{such that}\\
(\forall i\in\{1,\ldots,m\})\quad
z_i\in A_i\overline{x}_i+\Sum_{k=1}^KL_{ki}^*
\bigg(B_k\bigg(\Sum_{j=1}^mL_{kj}\overline{x}_j-r_k\bigg)\bigg),
\end{multline}
the dual problem 
\begin{multline}
\label{N09Kd424d}
\text{find}\;\;\overline{v}_1^*\in\GG_1,\ldots,\overline{v}^*_K
\in\GG_K
\;\;\text{such that}\\
(\forall k\in\{1,\ldots,K\})\quad
-r_k\in-\Sum_{i=1}^mL_{ki}\bigg(A_i^{-1}
\bigg(z_i-\Sum_{l=1}^KL_{li}^*\overline{v}^*_l\bigg)\bigg)
+B_k^{-1}\overline{v}^*_k,
\end{multline}
and the associated Kuhn-Tucker set
\begin{multline}
\label{ecnw7Gjh1107k}
\boldsymbol{Z}=\bigg\{(x_1,\ldots,x_m,v_1^*,\ldots,v^*_K)\in\KKK
\;\bigg |\;
(\forall i\in\{1,\ldots,m\})\;\;z_i-\sum_{k=1}^KL_{ki}^*v_k^*\in
A_ix_i\\
\text{and}\:\;(\forall k\in\{1,\ldots,K\})\;\;
\sum_{i=1}^mL_{ki}x_i-r_k\in B_k^{-1}v_k^*\bigg\}.
\end{multline}
The problem is to find a point in $\boldsymbol{Z}$.
The sets of solutions to \eqref{N09Kd424p} and 
\eqref{N09Kd424d} are denoted by $\mathscr{P}$ and 
$\mathscr{D}$, respectively.
\end{problem}

Such formulations, at least in their primal form 
\eqref{N09Kd424d}, have been investigated at various 
levels of generality in  
\cite{Sico10,Atto11,Bach12,Reic05,Bot13b,Nmtm09,Bric13,Cohe87,%
Jmiv11,Siop13,Fran12,Juan12} to model problems arising in 
areas such as 
game theory, evolution equations, machine learning, signal and 
image processing, mechanics, the cognitive sciences, 
and domain decomposition methods in partial differential equations.
As shown in \cite{Beck14} and \cite[Section~3]{Siop13}, 
another important motivation for 
studying such systems is the fact that single-variable inclusion
problems involving various types of parallel sums of monotone
operators, can be recast in the multivariate format 
\eqref{N09Kd424p} via the introduction of auxiliary 
variables. 

In this section, we shall use the following result, which 
establishes a bridge between Problem~\ref{prob:1} and 
Problem~\ref{prob:2}, to devise a splitting method for the 
latter based on Proposition~\ref{pmnmMyt^rw-19}.

\begin{proposition}
\label{pcnw7Gjh1110}
Consider the setting of Problem~\ref{prob:2} and set
\begin{equation}
\label{ecnw7Gjh1110a}
\begin{cases}
\HH=\bigoplus_{i=1}^m\HH_i\\ 
\GG=\bigoplus_{k=1}^K\GG_k\\ 
A\colon\HH\to 2^{\HH}\colon (x_i)_{1\leq i\leq m}\mapsto
\cart_{\!i=1}^{\!m}(-z_i+A_ix_i)\\
B\colon\GG\to 2^{\GG}\colon
(y_k)_{1\leq k\leq K}\mapsto\cart_{\!k=1}^{\!K}B_k(y_k-r_k)\\
L\colon\HH\to\GG\colon (x_i)_{1\leq i\leq m}\mapsto 
\big(\sum_{i=1}^mL_{ki}x_i\big)_{1\leq k\leq K}
\end{cases}
\end{equation}
in Problem~\ref{prob:1}. Then the following hold:
\begin{enumerate}
\item
\label{pcnw7Gjh1110i}
Problem~\ref{prob:1} coincides with Problem~\ref{prob:2}.
\item
\label{pcnw7Gjh1110ii}
Let $\gamma\in\RPP$, let $(x_i)_{1\leq i\leq m}\in\HH$, and 
let $(y_k)_{1\leq k\leq K}\in\GG$. Then 
\begin{multline}
J_{\gamma A}(x_i)_{1\leq i\leq m}=
\big(J_{\gamma A_i}(x_i+\gamma z_i)\big)_{1\leq i\leq m}\\
\text{and}\quad
J_{\gamma B}(y_k)_{1\leq k\leq K}=
\big(r_k+J_{\gamma B_k}(y_k-r_k)\big)_{1\leq k\leq K}.
\end{multline}
\end{enumerate}
\end{proposition}
\begin{proof}
\ref{pcnw7Gjh1110i}: This follows from \eqref{ecnw7Gjh1110a}
and the fact that $L^*\colon\GG\to\HH\colon (y_k)_{1\leq k\leq K}
\mapsto(\sum_{k=1}^KL_{ki}^*y_k)_{1\leq i\leq m}$.

\ref{pcnw7Gjh1110ii}: \cite[Propositions~23.15 and 23.16]{Livre1}.
\end{proof}

To find a Kuhn-Tucker point in Problem~\ref{prob:2} we can invoke 
Proposition~\ref{pcnw7Gjh1110} and apply the algorithms devised in 
Section~\ref{sec:3} in the setting of \eqref{ecnw7Gjh1110a}. Thus, 
Proposition~\ref{pmnmMyt^rw-19} leads to the following
result. 

\begin{theorem}
\label{tcnw7Gjh1107}
Consider the setting of Problem~\ref{prob:2}. Suppose that 
$\mathscr{P}\neq\emp$, let $\varepsilon\in\zeroun$, let 
$x_{1,0}\in\HH_1$, \ldots, $x_{m,0}\in\HH_m$, $v^*_{1,0}\in\GG_1$, 
\ldots, $v^*_{K,0}\in\GG_K$, and iterate
\begin{equation}
\label{ecnw7Gjh1107a}
\begin{array}{l}
\text{for}\;n=0,1,\ldots\\
\left\lfloor
\begin{array}{l}
(\gamma_n,\mu_n)\in [\varepsilon,1/\varepsilon]^2\\
\text{for}\;i=1,\ldots,m\\
\left\lfloor
\begin{array}{l}
a_{i,n}=J_{\gamma_n A_i}\Big(x_{i,n}+\gamma_n
\big(z_i-\sum_{k=1}^KL_{ki}^*v_{k,n}^*\big)\Big)\\
\end{array}
\right.\\[1mm]
\text{for}\;k=1,\ldots,K\\
\left\lfloor
\begin{array}{l}
l_{k,n}=\sum_{i=1}^mL_{ki}x_{i,n}\\
b_{k,n}=r_k+J_{\mu_n B_k}\big(l_{k,n}+\mu_nv_{k,n}^*-r_k\big)\\
t_{k,n}=b_{k,n}-\sum_{i=1}^mL_{ki}a_{i,n}\\
\end{array}
\right.\\[1mm]
\text{for}\;i=1,\ldots,m\\
\left\lfloor
\begin{array}{l}
s^*_{i,n}=\gamma_n^{-1}(x_{i,n}-a_{i,n})+
\mu_n^{-1}\sum_{k=1}^KL_{ki}^*(l_{k,n}-b_{k,n})\\
\end{array}
\right.\\[1mm]
\tau_n=\sum_{i=1}^m\|s_{i,n}^*\|^2+\sum_{k=1}^K\|t_{k,n}\|^2\\
\text{if}\;\tau_n=0\\
\left\lfloor
\begin{array}{l}
\text{for}\;i=1,\ldots,m\\
\left\lfloor
\begin{array}{l}
\overline{x}_{i}=a_{i,n}\\
\end{array}
\right.\\
\text{for}\;k=1,\ldots,K\\
\left\lfloor
\begin{array}{l}
\overline{v}^*_k=v_{k,n}^*+\mu_n^{-1}(l_{k,n}-b_{k,n})\\
\end{array}
\right.\\
\text{terminate}.
\end{array}
\right.\\
\text{if}\;\tau_n>0\\
\left\lfloor
\begin{array}{l}
\lambda_n\in\left[\varepsilon,2-\varepsilon\right]\\
\theta_n=\lambda_n\big(\gamma_n^{-1}\sum_{i=1}^m
\|x_{i,n}-a_{i,n}\|^2+\mu_n^{-1}
\sum_{k=1}^K\|l_{k,n}-b_{k,n}\|^2\big)/\tau_n\\
\text{for}\;i=1,\ldots,m\\
\left\lfloor
\begin{array}{l}
x_{i,n+1}=x_{i,n}-\theta_n s^*_{i,n}\\
\end{array}
\right.\\
\text{for}\;k=1,\ldots,K\\
\left\lfloor
\begin{array}{l}
v^*_{k,n+1}=v^*_{k,n}-\theta_n t_{k,n}.
\end{array}
\right.\\
\end{array}
\right.\\[8mm]
\end{array}
\right.\\[4mm]
\end{array}
\end{equation}
Then either \eqref{ecnw7Gjh1107a} terminates at a solution 
$(\overline{x}_1,\ldots,\overline{x}_m,\overline{v}_1^*,\ldots,
\overline{v}_K^*)\in\boldsymbol{Z}$ in a finite 
number of iterations or it generates infinite sequences 
$(x_{1,n})_{n\in\NN}$, \ldots, $(x_{m,n})_{n\in\NN}$,
$(v_{1,n}^*)_{n\in\NN}$, \ldots, $(v_{K,n}^*)_{n\in\NN}$
such that the following hold:
\begin{enumerate}
\item
\label{tcnw7Gjh1107i}
$(\forall i\in\{1,\ldots,m\})$
$\sum_{n\in\NN}\!\|s^*_{i,n}\|^2\!<\!\pinf$,
$\sum_{n\in\NN}\!\|x_{i,n+1}-x_{i,n}\|^2\!<\!\pinf$, and
$\sum_{n\in\NN}\!\|x_{i,n}-a_{i,n}\|^2\!<\!\pinf$.
\item
\label{tcnw7Gjh1107ii}
$(\forall k\in\{1,\ldots,K\})$
$\sum_{n\in\NN}\!\|t_{k,n}\|^2\!<\!\pinf$,
$\sum_{n\in\NN}\!\|v^*_{k,n+1}-v^*_{k,n}\|^2\!<\!\pinf$, and
$\sum_{n\in\NN}\!\|\sum_{i=1}^mL_{ki}x_{i,n}-b_{k,n}\|^2\!<\!\pinf$.
\item
\label{tcnw7Gjh1107iii}
For every $i\in\{1,\ldots,m\}$ $(x_{i,n})_{n\in\NN}$ converges
weakly to a point $\overline{x}_i$, for every $k\in\{1,\ldots,K\}$
$(v^*_{k,n})_{n\in\NN}$ converges weakly to a point
$\overline{v}_k^*$, 
$(\overline{x}_1,\ldots,\overline{x}_m)\in\mathscr{P}$, 
$(\overline{v}_1^*,\ldots,\overline{v}_K^*)\in\mathscr{D}$, and 
$(\overline{x}_1,\ldots,\overline{x}_m,\overline{v}_1^*,\ldots,
\overline{v}_K^*)\in\boldsymbol{Z}$.
\end{enumerate}
\end{theorem}
\begin{proof}
Define $\HH$, $\GG$, $A$, $B$, and $L$ as in \eqref{ecnw7Gjh1110a}.
Then, as seen in Proposition~\ref{pcnw7Gjh1110}\ref{pcnw7Gjh1110i}, 
Problem~\ref{prob:1} coincides with Problem~\ref{prob:2}. Now set 
\begin{equation}
(\forall n\in\NN)\quad
\begin{cases}
a_n=(a_{i,n})_{1\leq i\leq m}\\
s^*_n=(s^*_{i,n})_{1\leq i\leq m}\\
x_n=(x_{i,n})_{1\leq i\leq m}\\
b_n=(b_{k,n})_{1\leq k\leq K}\\
l_n=(l_{k,n})_{1\leq k\leq K}\\
t_n=(t_{k,n})_{1\leq k\leq K}\\
v^*_n=(v^*_{k,n})_{1\leq k\leq K}.
\end{cases}
\end{equation}
Then we derive from 
Proposition~\ref{pcnw7Gjh1110}\ref{pcnw7Gjh1110ii} that
\eqref{ecnw7Gjh1107a} coincides with \eqref{emnmMyt^rw-19d}. 
The assertions therefore follow from 
Proposition~\ref{pmnmMyt^rw-19}.
\end{proof}

\begin{remark}
\label{rcnw7Gjh1113}
\rm
In the special case when $m=1$, $A_1=0$, $z_1=0$, and, for every 
$k\in\{1,\ldots,K\}$, $\GG_k=\HH$, $L_{k1}=\Id$, and $r_k=0$, the 
primal problem \eqref{N09Kd424p} becomes
\begin{equation}
\label{eXc5t-iU8-19b}
\text{find}\;\;\overline{x}\in\HH\;\;\text{such that}\;\; 
0\in\sum_{k=1}^KB_k\overline{x},
\end{equation}
and the associated Kuhn-Tucker set of \eqref{ecnw7Gjh1107k} becomes
\begin{equation}
\label{eXc5t-iU8-19c}
\boldsymbol{Z}=\bigg\{(x,v_1^*,\ldots,v^*_K)\in\HH^{K+1}
\;\bigg |\;
\sum_{k=1}^Kv_k^*=0\:\;\text{and}\;\;
(\forall k\in\{1,\ldots,K\})\;\;v_k^*\in B_kx\bigg\}.
\end{equation}
In this setting, \eqref{ecnw7Gjh1107a} reduces to an algorithm
which is similar to that of \cite[Section~4]{Svai09}. The
convergence of the latter was established under the additional 
assumption that $\sum_{k=1}^KB_k$ is maximally monotone or 
that $\HH$ is finite-dimensional \cite[Proposition~4.2]{Svai09}, 
but these assumptions were subsequently shown not to be necessary 
\cite{Baus09}.
Let us note that in this special case, \eqref{ecnw7Gjh1107a} 
is different from the algorithm of \cite{Svai09} as it has 
a parallel structure (all the operators $(B_k)_{1\leq k\leq K}$ 
are used simultaneously) whereas that of \cite{Svai09} allows for
more flexibility (e.g., sequential activation) and it assigns to 
each monotone operator its own scaling parameter. It is natural to 
ask whether, in our general setting, \eqref{ecnw7Gjh1107a} could 
be extended to include such features by using 
Corollary~\ref{ccnw7Gjh1106} directly instead of 
Proposition~\ref{pmnmMyt^rw-19}. We have not been successful 
in bringing an affirmative answer to this question.
\end{remark}

\begin{remark}
\label{rx12vv2Frs15}
\rm
Using Proposition~\ref{pcnw7Gjh1110}, any algorithm for solving
Problem~\ref{prob:1} can in principle be used to solve 
Problem~\ref{prob:2}. However, methods which require
the computation of the norm of the operator $L$ of 
\eqref{ecnw7Gjh1110a} face the difficulty of expressing it
tightly in terms of those of the individual coupling operators
$(L_{ki})_{\substack{1\leq k\leq K\\ 1\leq i\leq m}}$;
see for instance \cite{Beck14,Siop13} for examples of such 
approximations.
This task is further complicated by the fact that in some
situations the norms of the individual coupling operators  
may not even be computable precisely. For
instance, in domain decomposition methods, $L_{ki}$
is the trace operator relative to the interface between two
subdomains and, depending on the underlying assumptions, its norm 
may not be easy to estimate. Likewise, inverting linear
operators based on various combinations of the individual coupling
operators is typically unfeasible in such applications, 
which renders inoperative those methods of 
\cite{Optl14,Bot13c,Siop11} using
such computations. These shortcomings of existing
methods are circumvented by \eqref{ecnw7Gjh1107a}, which makes it 
particularly attractive to solve Problem~\ref{prob:2}.
\end{remark}

\begin{remark}
\label{rcnw7Gjh1121}
\rm
An alternative method to solve Problem~\ref{prob:2} is that
proposed in \cite{Siop13}. In terms of complexity 
per iteration and parallelizability, both algorithms are quite 
comparable. However, as noted in Remark~\ref{rx12vv2Frs15},
the method of \cite{Siop13} implicitly requires a tight bound on 
the norm of the global operator $L$ of \eqref{ecnw7Gjh1110a}, 
which can be a serious drawback. Another difference
between the method of \cite{Siop13} and \eqref{ecnw7Gjh1107a}, is
that the latter features two sequences of scaling parameters 
$(\gamma_n)_{n\in\NN}$ and $(\mu_n)_{n\in\NN}$ which, furthermore, 
can be arbitrarily large or small. The proposed algorithm
\eqref{ecnw7Gjh1107a} also incorporates relaxation parameters 
$(\lambda_n)_{n\in\NN}$ that can induce large step sizes through
overrelaxations up to almost 2, whereas the method of 
\cite{Siop13} is unrelaxed. 
\end{remark}

An important area of application of Problem~\ref{prob:2} is 
multivariate convex minimization problems. We denote by 
$\Gamma_0(\HH)$ the class of lower semicontinuous convex proper
functions from $\HH$ to $\RX$. Let
$f\in\Gamma_0(\HH)$. The conjugate of $f$ is 
$\Gamma_0(\HH)\ni f^*\colon u\mapsto
\sup_{x\in\HH}(\scal{x}{u}-f(x))$.
For every $x\in\HH$, $f+\|x-\cdot\|^2/2$ has a unique minimizer, 
which is denoted by $\prox_fx$ \cite{Mor62b}. We have 
\begin{equation}
\label{e:prox2}
\prox_f=J_{\partial f},
\end{equation}
where
\begin{equation}
\partial f\colon\HH\to 2^{\HH}\colon x\mapsto
\menge{u^*\in\HH}{(\forall y\in\HH)\;\:\scal{y-x}{u^*}+f(x)
\leq f(y)} 
\end{equation}
is the subdifferential of $f$. The following formulation captures a
variety of multivariate minimization problems, e.g., 
\cite{Atto11,AujK06,Bach12,Nmtm09,Jmiv11,Cand11,Cohe87,%
Fran12,Juan12}.

\begin{problem}
\label{prob:3}
\rm
Let $m$ and $K$ be strictly positive integers, let
$(\HH_i)_{1\leq i\leq m}$ and $(\GG_k)_{1\leq k\leq K}$ be real 
Hilbert spaces, and set 
$\KKK=\HH_1\oplus\cdots\HH_m\oplus\GG_1\oplus\cdots\oplus\GG_K$.
For every $i\in\{1,\ldots,m\}$ and 
$k\in\{1,\ldots,K\}$, let $f_i\in\Gamma_0(\HH_i)$, let 
$g_k\in\Gamma_0(\GG_k)$, let $z_i\in\HH_i$, let $r_k\in\GG_k$, and 
let $L_{ki}\colon\HH_i\to\GG_k$ be linear and bounded. 
Let $\mathscr{P}$ be the set of solutions to the primal problem
\begin{equation}
\label{jji*yTTts23p}
\minimize{x_1\in\HH_1,\ldots,\,x_m\in\HH_m}{\sum_{i=1}^m
\big(f_i(x_i)-\scal{x_i}{z_i}\big)+\sum_{k=1}^K 
g_k\bigg(\sum_{i=1}^mL_{ki}x_i-r_k\bigg)},
\end{equation}
and let $\mathscr{D}$ be the set of solutions to the dual problem
\begin{equation}
\label{jji*yTTts23d}
\minimize{v^*_1\in\GG_1,\ldots,\,v^*_K\in\GG_K}{\sum_{i=1}^m
f_i^*\bigg(z_i-\sum_{k=1}^KL_{ki}^*v^*_k\bigg)
+\sum_{k=1}^K\big(g^*_k(v^*_k)+\scal{v^*_k}{r_k}\big)}.
\end{equation}
The problem is to find a point in the associated Kuhn-Tucker set
\begin{multline}
\label{ecnw7Gjh1113k}
\boldsymbol{Z}=\bigg\{(x_1,\ldots,x_m,v_1^*,\ldots,v^*_K)\in\KKK
\;\bigg |\;
(\forall i\in\{1,\ldots,m\})\;\;z_i-\sum_{k=1}^KL_{ki}^*v_k^*\in
\partial f_ix_i\\
\text{and}\:\;
(\forall k\in\{1,\ldots,K\})\;\;\sum_{i=1}^mL_{ki}x_i-r_k\in
\partial g_k^*v_k^*\bigg\}.
\end{multline}
\end{problem}

\begin{corollary}
\label{ccnw7Gjh1113}
Consider the setting of Problem~\ref{prob:3}. Suppose that 
\begin{equation}
\label{jji*yTTts21a}
(\forall i\in\{1,\ldots,m\})\quad
z_i\in\ran\bigg(\partial f_i+\sum_{k=1}^KL_{ki}^*\circ
\partial g_k\circ\bigg(\sum_{j=1}^mL_{kj}\cdot-r_k\bigg)\bigg),
\end{equation}
let $\varepsilon\in\zeroun$, let $x_{1,0}\in\HH_1$, \ldots, 
$x_{m,0}\in\HH_m$, let $v^*_{1,0}\in\GG_1$, \ldots, 
$v^*_{K,0}\in\GG_K$, and iterate \eqref{ecnw7Gjh1107a},
where the only modification is that we now set
\begin{equation}
\label{ecnw7Gjh1113b}
a_{i,n}=\prox_{\gamma_n f_i}\Bigg(x_{i,n}+\gamma_n
\Bigg(z_i-\sum_{k=1}^KL_{ki}^*v_{k,n}^*\Bigg)\Bigg)
\end{equation}
and
\begin{equation}
b_{k,n}=r_k+\prox_{\mu_n g_k}\big(l_{k,n}+\mu_nv_{k,n}^*-r_k\big).
\end{equation}
Then the conclusions of Theorem~\ref{tcnw7Gjh1107} hold true.
\end{corollary}
\begin{proof}
Set $(\forall i\in\{1,\ldots,m\})$ $A_i=\partial f_i$ and 
$(\forall k\in\{1,\ldots,K\})$ $B_k=\partial g_k$. Then, using 
the same arguments as in \cite[Proposition~5.4]{Siop13}, we 
obtain that \eqref{jji*yTTts23p} and \eqref{jji*yTTts23d}
are instances of \eqref{N09Kd424p} and \eqref{N09Kd424d}, 
respectively. 
\end{proof}

Sufficient conditions for this constraint qualification
\eqref{jji*yTTts21a} to hold 
can be found in \cite[Proposition~5.3]{Siop13}.
In particular, if $(\HH_i)_{1\leq i\leq m}$ and 
$(\GG_k)_{1\leq k\leq K}$ are finite-dimensional and if 
$\mathscr{P}\neq\emp$, then \eqref{jji*yTTts21a} is satisfied
if $(\forall i\in\{1,\dots,m\})
(\exi x_i\in\reli\dom f_i)(\forall k\in\{1,\ldots,K\})$
$\sum_{i=1}^mL_{ki}x_i-r_k\in\reli\dom g_k$.

\end{document}